\documentclass{amsart} 
\usepackage{amsmath, amssymb, euscript,  enumerate}

\newtheorem{thm}{Theorem}[section]
\newtheorem{cor}[thm]{Corollary} \newtheorem{lem}[thm]{Lemma}
\newtheorem{prop}[thm]{Proposition}
\newtheorem{claim}[thm]{Claim}
\newtheorem{defn}[thm]{Definition}
\theoremstyle{remark}
\newtheorem{rem}[thm]{Remark}
\numberwithin{equation}{section}

\theoremstyle{definition}

\newcommand{\R}{\mathbb R}

\newcommand{\ric}{\mathrm{Ric}}

\def\XXint#1#2#3{{\setbox0=\hbox{$#1{#2#3}{\intr}$}
     \vcenter{\hbox{$#2#3$}}\kern-.5\wd0}}

\begin{document}

\title{The classification of locally conformally flat Yamabe solitons}

\author{Panagiota Daskalopoulos$^*$}

\address{Department of Mathematics, Columbia University, New York, NY 10027,
 USA}
\email{pdaskalo@math.columbia.edu}

\author{Natasa Sesum$^{**}$}
\address{Department of Mathematics,  The State University of New Jersey, 110 Frelinghuysen Road, Piscataway, NJ 08854,
USA}
\email{natasas@math.columbia.edu}

\thanks{$*:$ Partially supported
by NSF grant 0905749}
\maketitle 

\begin{abstract}
We provide the classification of locally conformally flat gradient Yamabe solitons with positive sectional curvature. We first show that locally conformally flat gradient Yamabe solitons with positive sectional curvature have to be rotationally symmetric and then give the classification and asymptotic behavior of all radially symmetric gradient Yamabe solitons. We also show that any eternal solutions  to the Yamabe flow with positive Ricci curvature and with the scalar curvature attaining an  interior space-time maximum must be a steady Yamabe soliton.

\end{abstract}

\section{Introduction}

Our goal in this paper is to provide the classification of all complete locally conformally flat  Yamabe gradient  solitons.  

\begin{defn}\label{eqn-defn}  The metric  $g_{ij}$  is called a Yamabe
gradient soliton if there exists a smooth scalar (potential)  function $f: \R^n \to \R$ and a constant $\rho \in \mathbb{R}$, such 
that
\begin{equation}\label{eqn-ys}
(R - \rho) \, g_{ij} = \nabla_i\nabla_j f.
\end{equation}
If $\rho >0$, $\rho < 0$ or $\rho =0$, then $g$ is called a Yamabe shrinker, Yamabe expander or Yamabe steady
soliton respectively. As a matter of scaling the metric, we may assume with no loss of generality that $\rho = 1, -1, 0$, respectively.
\end{defn}

Yamabe solitons are special solutions to the Yamabe flow
\begin{equation}
\label{eqn-yf0}
\frac{\partial}{\partial t} g_{ij} = -Rg_{ij}.   
\end{equation}
This flow has been very well understood in the compact case and there is vast literature studying the compact Yamabe flow, such as \cite{Ch}, \cite{Y}, \cite{SS}, \cite{S1}, \cite{S2}, \cite{DS}. In \cite{S1} and \cite{S2} it has been showed that if $3\le n \le 5$ or if $n \ge 6$ (in the latter case Brendle imposes that the metric is either locally conformally flat  or he assumes a certain condition on the rate of vanishing of Weyl tensor at the points at which it vanishes), then starting at any initial metric, the normalized Yamabe flow has  long time existence and converges to a metric of constant scalar curvature. Therefore it is not surprising to expect 
that all compact Yamabe solitons to have constant scalar curvature. We actually prove that in Proposition \ref{prop-compact}. 

Unlike the compact Yamabe flow, which Brendle used to give another  proof of the Yamabe problem, the complete Yamabe flow is completely unsettled. In \cite{DSe} the authors showed that in the conformally flat case and
under certain conditions on the initial data, complete non-compact solutions to the Yamabe flow  develop
a finite time singularity and after re-scaling the metric converges to the Barenblatt solution (a certain type
of a shrinker, corresponding to the Type I singularity). The general case,  even when the solution is conformally  equivalent to 
$\R^n$,  is not well understood. 
  
Even though the analogue of Perelman's monotonicity formula is still lacking for the Yamabe flow, one expects  Yamabe soliton solutions to model finite time singularities of the Yamabe flow. This expectation has been justified in Corollary \ref{cor-eternal} below which says that in certain cases of Type II singularities we may expect steady Yamabe solitons to be the right singularity models.  The classification of Yamabe solitons is one of the important  steps in understanding the singularity formation in the complete Yamabe flow.

\begin{rem}\label{rem-par1}  (i) If $g_{ij}$ defines a  Yamabe shrinker   according to 
Definition \ref{eqn-defn}, then 
 the (time dependent) metric $\bar g_{ij} $  given  by
$$\bar g_{ij}(t) = (T-t)  \phi_t^*(g), \qquad  t <T$$
where   $\phi_t$ is an one parameter family of diffeomorphisms  generated by the vector field $X = \frac{\nabla f}{(T-t)}$, 
defines   an ancient  solution to the Yamabe flow \eqref{eqn-yf0} (also called a Yamabe shrinker) 
which vanishes at time $T$ and  satisfies
$$(\bar R - \frac {1}{T-t}) \,  \bar g_{ij}(t) = \nabla_i \nabla_j f.$$  

(ii) Similarly, if  $g_{ij}$ is a Yamabe expander  then 
the (time dependent) metric $\bar g_{ij}$ defined by
$$\bar g_{ij}(t) = t \,  \phi_t^*(g_{ij}), \qquad  t >0,$$
where  $\phi_t$ is an one parameter family of diffeomorphisms generated by $X = \frac{\nabla f}{t}$,
is a solution to the Yamabe flow \eqref{eqn-yf0} (also called a Yamabe expander)
which is defined on $0 < t < \infty$ and  satisfies
$$(\bar R - \frac {1}{t}) \,  \bar g_{ij}(t) = \nabla_i \nabla_j f.$$ 
\end{rem}

(iii) Finally, if $g_{ij}$ defines  a Yamabe steady soliton according to 
Definition \ref{eqn-defn},  then 
the (time dependent) metric $\bar g_{ij}$ defined by
$$\bar g_{ij}(t) =  \phi_t^*(g_{ij}), \qquad -\infty <  t  < \infty,$$
where  $\phi_t$ is an  one parameter family of diffeomorphisms generated by $\nabla f$,
is an eternal solution to the Yamabe flow \eqref{eqn-yf0} (also called a Yamabe steady soliton)
which   satisfies
$$\bar R  \,  \bar g_{ij}(t) = \nabla_i \nabla_j f.$$

Our first result establishes the rotational symmetry of locally conformally flat Yamabe solitons.

\begin{thm}[Rotational symmetry of Yamabe solitons]
\label{thm-unique}
All locally conformally flat complete Yamabe gradient  solitons  with positive 
sectional  curvature   have to be rotationally symmetric.
\end{thm}
 
We will show at the end of section \ref{sec-rs}  that the result in \cite{CH} implies    that rotationally symmetric  complete 
Yamabe   solitons with nonnegative sectional curvature 
are globally conformally flat, namely $g_{ij} = u^{\frac 4{n+2}} \, dx^2$, where $dx^2$
denotes the standard metric on $\R^n$ and $u^{\frac 4{n+2}}$ is the conformal factor. 
We  have the following result.

\begin{prop} [PDE formulation of Yamabe  solitons]
\label{prop-ell00} Let $g_{ij}=u^{\frac 4{n+2}}\, dx^2$ be a  conformally flat  rotationally symmetric Yamabe gradient soliton  with positive sectional  curvature.   
Then, $u$ is a smooth solution of the elliptic equation
\begin{equation}\label{eqn-ell00}
\frac{n-1}{m} \, \Delta u^m + \beta  \, x \cdot \nabla u  +  \gamma \, u =0,
\qquad \mbox{on}\,\,\, \R^n
\end{equation}
where $\beta \geq 0$ and 
$$\gamma = \frac{2\beta + \rho}{1-m}, \qquad m = \frac{n-2}{n+2}.$$
In the case of the expanders, $\beta > 0$.
In addition, any smooth solution of the elliptic equation \eqref{eqn-ell00} with $\beta$ and $\gamma$
as above defines a gradient  Yamabe soliton.
\end{prop}

The above Proposition reduces the classification of Yamabe solitons to
the classification of global smooth solutions of the elliptic equation \eqref{eqn-ell00}.

To simplify the notation, we will assume from now on that $\rho = 1$ in \eqref{eqn-ys}
(and hence in Proposition \ref{prop-ell00} as well)
in the case of the Yamabe shrinkers, and that  $\rho =-1$ in the case of the Yamabe expanders.  
This can be easily achieved by scaling our metric $g$. 

The following result provides the classification of radially symmetric  and smooth
solutions of the elliptic equation \eqref{eqn-ell00}. 

\begin{prop}[Classification of radially symmetric Yamabe solitons] \label{prop-cys}
Let  $m=\frac{n-2}{n+2}$. The elliptic equation \eqref{eqn-ell00} admits non-trivial radially symmetric  smooth  solutions
if and  only if $\beta \geq 0$ and $\gamma=\frac{2\beta +\rho}{1-m} >0$.
More precisely, we have:
\begin{enumerate}[i.]

\item  {\bf Yamabe shrinkers $\rho=1$:} 
For any $\beta >0$ and $\gamma=\frac{2\beta +1}{1-m}$,
there exists an one parameter
family $u_\lambda$, $\lambda >0$,  of smooth radially symmetric solutions of equation \eqref{eqn-ell00} on $\R^n$  of  slow-decay  rate at infinity, namely 
$u_\lambda(x) = O (|x|^{- \frac 2{1-m}})$, as $|x| \to \infty$. 
We will refer to them as cigar solutions. In the case 
 $\gamma  = \beta  n$ the solutions are given in the closed form
 \begin{equation}\label{eqn-ub}
u_\lambda (x) = \left ( \frac {C_n }{\lambda^2 + |x|^2} \right )^{\frac 1{1-m}}, \qquad C_n=(n-2)(n-1)
\end{equation}
and will refer to them as the Barenblatt solutions. 
When $\beta =0$ and $\gamma=\frac 1{1-m}$ equation \eqref{eqn-ell00} admits the explicit
solutions of  fast-decay rate
 \begin{equation}\label{eqn-shr}
u_\lambda (x) = \left ( \frac {C_n \, \lambda}{\lambda^2 + |x|^2} \right )^{\frac 2{1-m}}, \qquad C_n=
(4n (n-1))^\frac 12. 
\end{equation}
We will refer to them as the spheres.

\item {\bf Yamabe expanders  $\rho=-1$:} 
For any $\beta > 0$ and $\gamma=\frac{2\beta - 1}{1-m} > -\frac{1}{1-m} $,
there exists an one parameter
family $u_\lambda$, $\lambda >0$,  of smooth radially symmetric solutions of equation \eqref{eqn-ell00} on $\R^n$. 

\item {\bf Yamabe steady solitons  $\rho=0$:} For any $\beta >0$ and $\gamma 
=\frac{2\beta }{1-m}>0$, there exists an  
one parameter family 
$u=u_{\lambda}$, $\lambda >0$,  of smooth solutions of equation \eqref{eqn-ell00} on $\R^n$
which satisfy the asymptotic behavior
$u= O (( \frac{\log |x|}{|x|^2})^{\frac 1{1-m}})$,  as $|x| \to \infty$. We will refer to them as logarithmic cigars.  If $\beta = 0$ and therefore $\gamma = 0$, then $u$ is a constant, defining the euclidean metric on $\mathbb{R}^n$.
\end{enumerate}

In all of the above cases the solution $u_\lambda$ is uniquely determined by its value at the origin.  

\end{prop}

\begin{rem}\label{rem-par2} [Self-similar solutions of the fast-diffusion equation] There is a clear
connection between Yamabe solitons and self-similar solutions of the
fast diffusion equation
\begin{equation}\label{eqn-fde}
\frac{\partial \bar u}{\partial t} =  \frac{n-1}{m} \Delta \bar u^m, \qquad m=\frac{n-2}{n+2}.
\end{equation}

(i) {\em Yamabe shrinkers $\rho >0$: } The function $u$ is a solution of the
elliptic equation \eqref{eqn-ell00} if and only if 
$\bar u (x,t)= (T-t)^\gamma \, u(x, (T-t)^\beta)$ is an ancient  solution of \eqref{eqn-fde} which
vanishes at $T$. The existence of such solutions is proven in \cite{Vaz1} (Proposition 7.4)
and it was also noted  in \cite{GP}.

(ii) {\em Yamabe expanders $\rho <0$: } The function $u$ is a solution of the
elliptic equation \eqref{eqn-ell00} if and only if 
$\bar u (x,t)= t^{-\gamma} \, u(x, t^{-\beta})$ is a  solution of \eqref{eqn-fde} which
is defined for all  $0 < t < \infty$. 

(iii) {\em Yamabe steady solitons $\rho =0$: } The function $u$ is a solution of the
elliptic equation \eqref{eqn-ell00} if and only if 
$\bar u(x,t)  = e^{-\gamma t}  \, u(x, e^{-\beta  t})$ is an eternal  solution of \eqref{eqn-fde}. 
The existence of such solutions (without a proof)  was first noted in \cite{GP}.

In all of the above cases, $\bar g (t) = u^{\frac 4{n+2}}(\cdot,t)$ defines a solution of the
Yamabe flow \eqref{eqn-yf0}.\end{rem}

Combining the above results leads to the following classification of Yamabe solitons.

\begin{thm} The metric $g$ is a  complete  locally conformally flat Yamabe gradient  soliton 
with positive  sectional curvature  if and only if 
$g=u^{\frac 4{n+2}} \, dx^2$,  where $u$ satisfies the elliptic equation \eqref{eqn-ell00}, 
for some  $\beta \geq 0$ and $\gamma := \frac{2\beta +\rho}{1-m} >0$. The classification
of all such metrics is given in Proposition \ref{prop-cys}.
\end{thm}

In the case of compact gradient Yamabe solitons we have the following more general result:

\begin{prop}\label{prop-compact}
If $(M,g,f)$ is a compact gradient Yamabe soliton, not necessarily 
locally conformally flat,  
then $g$ is the metric of constant scalar curvature.  \end{prop}

Note that in this result we do not make any assumptions on the sign of sectional curvatures. 

{\bf Acknowledgements:} We would like to thank Robert Bryant and Richard Hamilton for many useful discussions.

\section{Yamabe solitons are Rotationally symmetric}\label{sec-rs}

In this section we will establish the rotational symmetry of 
locally conformally flat Yamabe solitons with nonnegative sectional curvature, 
Theorem \ref{thm-unique}. 
Our  proof is  inspired by the proof of the analogous theorem for complete gradient steady Ricci solitons in \cite{CC} by Cao and Chen.

\begin{proof}[Proof of Theorem \ref{thm-unique}]
We will first deal with the case of steady solitons  
\begin{equation}
\label{eq-sol100}
R\, g_{ij} = \nabla_i\nabla_j f
\end{equation}
 where we refer to $f$ as to a potential function. 
 The other two cases of shrinkers and expanders  can be treated in the same way as it will be explained at the end of the proof. 
 Since $R > 0$, the potential function $f$ is strictly convex and therefore it  has at most one critical point. Denote  by $G = |\nabla f|^2$  
 and observe that  in any neighborhood, where $G \neq 0$, of the level  surface 
$$\Sigma_c := \{x\in M : f(x) = c\}$$
for  a regular value $c$ of $f$,  we can express the metric $g$ as
\begin{equation}
\label{eq-met-new}
g = \frac{1}{G(f,\theta)}df^2 + g_{ab}(f,\theta)\, d\theta^a\, d\theta^b
\end{equation}
where $(\theta^2, \dots, \theta^n)$ denote  intrinsic coordinates for $\Sigma_c$. 

We wish  to show that $G = G(f)$, $g_{ab} = g_{ab}(f)$, and that $(\Sigma_c,g_{ab})$ is a space form of positive constant curvature. This  would mean that $g$ has the  form 
\begin{equation}
\label{eq-metric-form}
g = \psi^2(f) \, df^2 + \phi^2(f) \, g_{S^{n-1}},
\end{equation}
where $g_{S^{n-1}}$ denotes the standard metric on the unit sphere $S^{n-1}$. As in \cite{CC} it can be argued that $f$ has exactly one critical point, leading to the fact that $g$ is a rotationally symmetric metric on $\mathbb{R}^n$.

Next we derive some identities on Yamabe solitons that will be used later in the paper.

\begin{lem}
If $G := |\nabla f|^2$, then
\begin{equation}
\label{eq-G}
\nabla G = 2R\, \nabla f.
\end{equation}
Furthermore,
\begin{equation}
\label{eq-yam-id}
(n-1)\nabla R = \ric \, (\nabla f, \cdot).
\end{equation}
\end{lem}

\begin{proof}
Fix $p\in M$ and  choose  normal coordinates around $p$ so that the metric matrix is diagonal at $p$. Then,
$$\nabla_i G = 2\, \nabla_i\nabla_j f\,  \nabla_j f = 2Rg_{ij}\nabla_j f$$
implying $\nabla_i G = 2R\, \nabla_i f$. In other words,
$$\nabla G = 2R\, \nabla f.$$

Moreover,  continuing to compute  in normal coordinates around $p\in M$, if we apply $\nabla_k$ to our soliton equation $\nabla_i\nabla_j f = R\, g_{ij}$ we obtain
$$\nabla_k\nabla_i\nabla_j f = \nabla_k R \,  g_{ij}$$
implying that 
$$\nabla_i\nabla_k\nabla_j f + R_{kil}^j\nabla^l f = \nabla_k R\,  g_{ij}.$$
Tracing  the previous equation in $k$ and $j$, we obtain 
$$\nabla_i \Delta f + R_{il}\nabla^l f = \nabla_i R.$$
On the other hand, after tracing the soliton equation we get
$$\Delta f = nR$$
and therefore
$$n\nabla_i R + R_{il}\nabla^l f = \nabla_i R.$$
We conclude that  the following identity  holds on any Yamabe steady soliton:
$$(n-1)\nabla R = \ric \, (\nabla f, \cdot).$$
\end{proof}

In the following Proposition we will show that  the Ricci tensor of our steady soliton metric $g$ has only two distinct eigenvalues. Cao and Chen  proved  the same  theorem in \cite{CC}  using  the properties of the  Cotton tensor together with the Ricci soliton equation. Our proof 
uses  the Harnack expression for the Yamabe flow that has been introduced by Chow in \cite{Ch}.

\begin{prop}
\label{prop-ricci}
At any point $p\in \Sigma_c$, the Ricci tensor of $g$ has either a unique eigenvalue $\lambda$, or 
it has two distinct eigenvalues $\lambda$ and $\mu$, of multiplicity $1$ and $n-1$ respectively. In either case, $e_1 = \frac{\nabla f}{|\nabla f|}$ is an eigenvector with eigenvalue $\lambda$. Moreover, for any orthonormal basis $e_2, \dots e_n$ tangent to the level surface $\Sigma_c$ at $p$, we have
\begin{enumerate}[i.]
\item
$\ric(e_1,e_1) = \lambda$
\item
$\ric(e_1,e_b) = R_{1b} = 0, \,\,\,\, b = 2, \dots n$
\item
$\ric(e_a,e_b) = R_{aa}\delta_{ab}, \,\,\, a,b = 2, \dots, n$,
\end{enumerate}
where either $R_{11} = \dots R_{nn} = \lambda$ or $R_{11} = \lambda$ and $R_{22} = \dots = R_{nn} = \mu$.
\end{prop}

The proof of Proposition \ref{prop-ricci} will make use of the evolution of the 
Harnack expression for the scalar curvature, which has been introduced  by Chow in \cite{Ch}.  We will compute its evolution and express it   in a form that is convenient for our purposes. This computation does not depend on having the soliton equation, but only on evolving the metric by the Yamabe flow. 

Assume that we have a complete eternal  locally conformally flat Yamabe flow
\begin{equation}\label{eqn-yf}
g_t = - R\, g, \qquad - \infty < t < +\infty
\end{equation}
where $g$ has  positive Ricci curvature. 
Choose a vector field $X$ to
satisfy
\begin{equation}
\label{equation-vf}
\nabla_i R + \frac{1}{n-1}R_{ij}X_j = 0. 
\end{equation}
The vector field $X$ is well defined since $\ric > 0$ (and therefore
defines an invertible matrix). Following Chow \cite{Ch} we define the
Harnack expression for the eternal Yamabe flow, namely 
\begin{equation}
\label{equation-har}
Z(g,X) = (n-1)\Delta R + \langle\nabla
R, X\rangle + \frac{1}{2(n-1)}\,  R_{ij}X_iX_j + R^2.
\end{equation}
Note that in \eqref{equation-har} we have  dropped  the term $\frac{R}{t}$, due to the fact we have a  solution that is defined up to $t=-\infty$.

To simplify the notation, we define $\Box = \partial_t - (n-1)\, \Delta$. 

\begin{lem}
\label{lemma-evolution-Z} The quantity $Z$ defined by \eqref{equation-har} evolves by 
\begin{equation}
\label{eq-evolution-Z}
\Box{Z} = RZ + A_{ij}X_iX_j + g^{kl}R_{ij}(Rg_{ik} -
\nabla_iX_k)(Rg_{jl} - \nabla_jX_l)
\end{equation}
where $A_{ij}$ is the same
matrix that Chow defines by $(3.13)$ in \cite{Ch}.
\end{lem}

\begin{proof}
We have the following equation due to Chow (\cite{Ch}) after
dropping all terms with $1/t$:
\begin{equation}
\begin{split}
\label{equation-Z}
\Box{Z} &= 3RZ - R^3 + \frac{1}{2}(R_{kijl} - R_{ij}^2)X_iX_j \\
&- \frac{1}{2(n-1)}RR_{ij}X_iX_j - \frac{(n-1)(n-2)}{2}|\nabla
R|^2 \\
&- (n-1)R_{ij}\nabla_i R \, X_j + \langle\nabla R, \Box{X}\rangle +
\frac{R_{ij}X_i}{n-1} \, \Box{X_j}\\
&- 2R_{ij}\nabla_k X_i\nabla_k X_j - 2\nabla_kR_{ij}\nabla_k X_i
X_j - 2(n-1)\langle\nabla\nabla R,\nabla X\rangle \\
&+ R_{ij}\nabla_kX_i\nabla_kX_j.
\end{split}
\end{equation}
Since the evolution equation for $Z$ is independent of the choice of coordinates,
choose the coordinates at a point at which $g_{ij} = \delta_{ij}$ and the
Ricci tensor is diagonal at that point. By  (\ref{equation-vf}), we have 
\begin{equation}
\label{equation-1}
\langle \nabla R, \Box{X}\rangle +
\frac{R_{ij} X_i}{n-1}\, \Box{X_j} = [\nabla_j R +
\frac{R_{ij}X_i}{n-1}]\, \Box{X_j} = 0
\end{equation}
and  
\begin{equation}
\begin{split}
\label{equation-2}
&\, 2R_{ij}\nabla_kX_i\nabla_kX_j +
2\nabla_kR_{ij}\nabla_k X_iX_j + 2(n-1)\langle\nabla\nabla R,
\nabla X\rangle \\
&= 2(\nabla_kX_i\cdot\nabla_k(R_{ij}X_j) +
(n-1)\nabla_kX_i\nabla_k\nabla_iR) \\
&= 2\nabla_kX_i\nabla_k(R_{ij}X_j + (n-1)\nabla_i R) = 0. 
\end{split}
\end{equation}

Combining (\ref{equation-Z}), (\ref{equation-1}) and
(\ref{equation-2}) yield to the equation 
\begin{equation}
\begin{split}
\label{equation-Z1}
\Box{Z} &= 3RZ - R^3 +
\frac{1}{2}(R_{kijl}R_{kl} - R_{ij}^2)X_iX_j -
\frac{1}{2(n-1)}RR_{ij}X_iX_j \\
&- \frac{(n-1)(n-2)}{2}|\nabla R|^2 - (n-1)R_{ij}\nabla_iR \, X_j +
R_{ij}\nabla_kX_i\nabla_kX_j.
\end{split}
\end{equation}
We recall the following  basic identity which holds  for locally conformally flat manifolds 
$$
R_{kijl} = \frac{1}{n-2}(R_{kl}g_{ij} +
R_{ij}g_{kl} - R_{kj}g_{il} - R_{il}g_{kj} -
\frac{R\, (g_{kl}g_{ij} - g_{kj}g_{il})}{n-1}) \\
$$
If we contract this identity  by $R_{kl}$ we get at a point where $g_{ij}=\delta_{ij}$
and $R_{ij}$ is also diagonal
\begin{equation*}
\begin{split}
R_{kijl}R_{kl}  
&= \frac{1}{n-2}(|\ric|^2 \, \delta_{ij} + R_{ij}R \, \delta_{ij} -
R_{ij}^2\delta_{ij} - R_{ij}^2\delta_{ij} -
\frac{R}{n-1}(R\, \delta_{ij} - R_{ij})) \\
&= \frac{1}{n-2} \left ( \, |\ric|^2 + \frac{n}{n-1}RR_{ij} - 2R_{ij}^2 -
\frac{R^2}{n-1}\, \right )\, \delta_{ij}
\end{split}
\end{equation*}
and therefore
\begin{equation}
\label{eq-curv}
\frac{1}{2}(R_{kijl}R_{kl} - R_{ij}^2) =
\frac{1}{2(n-2)}\, \left ( |\ric|^2 + \frac{n}{n-1}RR_{ij} - nR_{ij}^2 -
\frac{R^2}{n-1} \right )\,\delta_{ij}.
\end{equation}
We also have $\nabla_i R = -\frac{1}{n-1}R_{ij}X_j$, hence 
\begin{equation}
\begin{split}
\label{eq-scalar}
|\nabla R|^2 &= g^{ij}\nabla_i R\nabla_j R =
\frac{1}{(n-1)^2}R_{ik}X_kR_{jl}X_l \\
&= \frac{1}{(n-1)^2}R_{ik}R_{il}X_kX_l\delta_{ik}\delta_{il} = \frac{1}{(n-1)^2}R_{ij}^2X_iX_j\delta_{ij}
\end{split}
\end{equation}
and
\begin{equation}
\label{eq-mix}
-(n-1)R_{ij}\nabla_iRX_j = R_{ik}X_kR_{ij}X_j =
R_{ij}^2X_iX_j\delta_{ij}.
\end{equation}
Combining (\ref{equation-Z1}), (\ref{eq-curv}) and (\ref{eq-mix})
yield to the equation 
\begin{equation}
\label{eq-Z2}
\Box{Z} = 3RZ - R^3 + A_{ij}X_iX_j\delta_{ij} +
R_{ij}\nabla_kX_i\nabla_kX_j.
\end{equation}
where 
$$A_{ij} = \frac{1}{n-2} \, \left ( \frac{|\ric|^2}{2} + \frac{RR_{ij}}{n-1} -
\frac{R^2}{2(n-1)} - \frac{n}{2(n-1)}R_{ij}^2 \right ) \, g_{ij}.$$ 
Direct computation gives 
\begin{equation}
\begin{split}
\label{eq-rest}
2RZ - R^3 &= 2R \left ( (n-1)\Delta R + \langle \nabla
R,X\rangle +
\frac{1}{2(n-1)}R_{ij}X_iX_j + R^2 \right ) - R^3 \\
&= 2(n-1)R\, \Delta R + 2R\, \langle\nabla R, X\rangle +
\frac{RR_{ij}}{n-1}X_iX_j + R^3.
\end{split}
\end{equation}
Also, after  taking the covariant derivative $\nabla_k$ of
(\ref{equation-vf}), we find that 
\begin{equation}\label{eq-der}
\nabla_k \nabla_i R + \frac{1}{n-1}\nabla_k R_{ij} X_j
+ \frac{1}{n-1}R_{ij}\nabla_k X_j = 0
\end{equation}
which gives
$$\nabla_k \nabla_i R = - \frac 1{n-1} \left (  \nabla_k R_{ij} X_j
+ R_{ij}\nabla_k X_j   \right ).$$
 
If we sum (\ref{eq-der}) over $i$, by the contracted Bianchi
identity $\nabla_i R_{ij} = \frac{1}{2}\nabla_j R$,  we get
$$\Delta R + \frac{1}{2(n-1)}\nabla_j R \, X_j +
\frac{1}{n-1}R_{ij}\nabla_iX_j = 0.$$
By (\ref{equation-vf}) and
the previous identity we have
$$2(n-1)R\Delta R = \frac{RR_{ij}}{n-1}X_iX_j - 2RR_{ij}\nabla_iX_j$$
which combined with (\ref{eq-rest}) yields
\begin{equation*}
\begin{split}
2RZ - R^3 &= \frac{RR_{ij}}{n-1}X_iX_j -
2RR_{ij}\nabla_iX_j - \frac{2RR_{ij}}{n-1}X_iX_j +
\frac{RR_{ij}}{n-1}X_iX_j + R^3 \\
&= R^3 - 2RR_{ij}\nabla_iX_j.
\end{split}
\end{equation*}
It follows from \eqref{eq-Z2} that
\begin{equation}
\label{eq-Z3}
\Box{Z} = RZ + A_{ij}X_iX_j\delta_{ij} +
(2RZ - R^3 + R_{ij}\nabla_kX_i\nabla_kX_j)
\end{equation}
where by the discussion above 
\begin{equation*}
I:=2RZ - R^3 +  R_{ij}\nabla_kX_i\nabla_kX_j = R^3 - 2RR_{ij}\nabla_i X_j +
R_{ij}\nabla_kX_i\nabla_kX_j. 
\end{equation*}
Hence,  at the chosen coordinates at a point 
where $g_{ij} = \delta_{ij}$ and $R_{ij}$ is diagonal, we have
\begin{equation}
\begin{split}\label{equation-rest1}
I &= R^3 - 2RR_{ij}\nabla_i X_j +
R_{ij}\nabla_kX_i\nabla_kX_j = 
\sum_{i}R_{ii} \, (R^2 g_{ii} - 2R\nabla_iX_i +  |\nabla_iX_i |^2) +  \sum_{i\neq k} R_{ii} |\nabla_kX_i |^2 \\
&= \sum_{i}R_{ii} \, (R\, g_{ii} - \nabla_iX_i)^2 + \sum_{i\neq k} R_{ii}\,|\nabla_kX_i |^2
=g^{kl}R_{ij}(R\, g_{ik} -
\nabla_iX_k)(R\, g_{jl} - \nabla_jX_l). 
\end{split}
\end{equation}
By combining \eqref{eq-Z3} with \eqref{equation-rest1} we readily conclude \eqref{eq-evolution-Z}. 
The matrix $A_{ij}$ is the same   that Chow defines by $(3.13)$ in \cite{Ch}).
In local coordinates $\{x_i\}$,  where $g_{ij} =\delta_{ij}$ and the Ricci tensor is
diagonal at a  point, we have 
\[ R_{ij} = \left( \begin{array}{ccc}
\lambda_1 & &  \\
& \ddots &   \\
& & \lambda_n\\
\end{array} \right)\] 
hence
\begin{equation}
\label{eq-A-chow}
A_{ij} = \left( \begin{array}{ccc}
\nu_1 & &  \\
& \ddots &   \\
& & \nu_n \\
\end{array} \right)
\end{equation}
where
$$\nu_i = \frac{1}{2(n-1)(n-2)}\sum_{k,l\neq i, k> l} (\lambda_k - \lambda_l)^2.$$
\end{proof}

We will now give the proof of Proposition \ref{prop-ricci}.

\begin{proof}[Proof of Proposition \ref{prop-ricci}]
Assume now that the solution \eqref{eqn-yf} is 
a steady soliton, namely it satisfies \eqref{eq-sol100}.
 Taking the divergence of the above equation,  tracing and then taking the Laplacian yields to (see \cite{Ch} for details) 
$$(n-1)\Delta R + \frac{1}{2}\langle\nabla R, \nabla f\rangle + R^2 = 0.$$
With our choice of $X$ in (\ref{equation-vf}) we have that
$Z(g,X) = 0$, 
if $g$ is a steady Yamabe soliton. Then form (\ref{eq-evolution-Z}) we find
$$A_{ij}\nabla_i\nabla_j f \equiv 0, \qquad \mbox{on} \,\,\, M.$$
Then (\ref{eq-A-chow}) implies that at every point $p\in M$ either all eigenvalues of Ricci tensor $\lambda_1 = \dots = \lambda_n = \lambda$ are the same, or there are two distinct eigenvalues $\lambda$ and $\mu$ with multiplicities $1$ and $n-1$ respectively. In the latter case, say $\nabla_1 f \neq 0$ and $\nabla_i  f = 0$  for $i = 2, \dots n$, then $\nabla f = |\nabla f|\, e_1$, with $e_1 = \frac{\nabla f}{|\nabla f|}$  an eigenvector of Ricci tensor  and $\lambda_2 = \dots =\lambda_n$. In either case, we conclude that $\nabla f$ is an eigenvector of $\ric$. Other properties of $\ric$ listed in the statement of Proposition \ref{prop-ricci} now easily follow.
\end{proof}

To conclude the proof of Theorem \ref{thm-unique} we need the following lemma.

\begin{lem}
\label{lem-properties}
Let $c$ be a regular value of $f$ and $\Sigma_c = \{f = c\}$. Then,
\begin{enumerate}[i.]
\item
The function $G = |\nabla f|^2$ and the scalar curvature $R$ are constant on $\Sigma_c$, that is, they are functions of $f$ only.
\item
The mean curvature $H$ of $\Sigma_c$ is constant.
\item
The sectional curvature of the induced metric on $\Sigma_c$ is constant.
\end{enumerate}
\end{lem}

\begin{proof}
Let $\{e_1, e_2, \dots e_n\}$ be an orthonormal frame with $e_1 = \frac{\nabla f}{|\nabla f|}$ and $e_2, \dots e_n$ tangent to $\Sigma_c$. By (\ref{eq-G}) we have
\begin{equation}
\label{eq-G-ind}
\nabla_a G = 2R\, \nabla_a f = 0, \qquad  a = 2, \dots n,
\end{equation}
since $e_a, \,\,\, i = 2, \dots n$ are tangential directions to the level surfaces $\Sigma_c$ on which $f$ is constant. Furthermore, using (\ref{eq-yam-id})  and Proposition \ref{prop-ricci} we get
\begin{equation}
\label{eq-R-ind}
(n-1)\nabla_a R = \ric(\nabla f, e_a) = 0, \,\,\, a = 2, \dots, n.
\end{equation}
Observe that (\ref{eq-G-ind}) and (\ref{eq-R-ind}) prove part (i) of our Lemma.

The second fundamental form of the level surface $\Sigma_c$ is given by 
$$h_{ab} = \frac{f_{ab}}{\sqrt{G}} = \frac{R\, g_{ab}}{\sqrt{G}} = \frac{H\, g_{ab}}{n-1}$$
where $H = \frac{(n-1)R}{\sqrt{G}}$ is the mean curvature of hypersurface $\Sigma_c$.
By part (i),  both  G and H are constant on $\Sigma_c$ and therefore the mean curvature $H$ of $\Sigma_c$ is constant. This proves (ii).

It remains to show that (iii) holds. By the Gauss equation, the sectional curvatures of $(\Sigma_c, g_{ab})$ are given by
\begin{equation}
\label{eq-sect-surf}
R^{\Sigma_c}_{abab} = R_{abab} + h_{aa}h_{bb} - h_{ab}^2 = R_{abab} + \frac{H^2}{(n-1)^2}.
\end{equation}
Since $W_{ijkl} = 0$,  we get
\begin{equation}
\label{eq-weyl-zero}
R_{ijkl} = \frac{1}{n-2}(g_{ik} R_{jl} - g_{il}R_{jk} - g_{jk}R_{il} + g_{jl}R_{ik}) - \frac{R}{(n-1)(n-2)}(g_{ik}g_{jl} - g_{il}g_{jk}).
\end{equation}
Using (\ref{eq-weyl-zero}) and Proposition \ref{prop-ricci}  we obtain
\begin{equation}
\label{eq-sect}
R_{abab} = \frac{2}{n-2}R_{aa} - \frac{R}{(n-1)(n-2)} = \frac{R - 2R_{11}}{(n-1)(n-2)}.
\end{equation}
Our goal is to show that $\nabla_a R_{aa} = 0$, that is, $R_{aa}$ is constant on the level surface  $\Sigma_c$. This together with $R$ and $H$ being constant on $\Sigma_c$ will yield to the constancy of sectional curvatures of $\Sigma_c$. 

Recall that our metric $g$ can be expressed as
$g = \frac{1}{G(f)} df^2 + h_{ab}(f,\theta)d\theta^a\,d\theta^b,$
where $(f,\theta_2, \dots, \theta_n)$ are the local coordinates on our soliton and $(\theta_2, \dots, \theta_n)$ are the intrinsic coordinates for $\Sigma_c$. Performing the computation in local coordinates we find
$$
\nabla_a\nabla_1 R = \frac{\partial^2 R}{\partial \theta_a\partial f}  - \sum_{l}\Gamma_{a1}^l\nabla_l R 
= \frac{\partial^2 R}{\partial \theta_a\partial f}  - \Gamma_{a1}^1\nabla_1 R
$$
since $R = R(f)$. Furthermore, if we call $\theta_1 := f$, using that $g_{1a} = \frac{1}{G(f)}\delta_{1a}$, then
\begin{eqnarray*}
\Gamma_{a1}^1 &=& \frac{g^{1k}}{2}\, \left(\frac{\partial g_{ak}}{\partial f} + \frac{\partial g_{1k}}{\partial \theta_a} - \frac{\partial g_{a1}}{\partial \theta_k}\right) \\
 &=& \frac{g^{11}}{2}\, \left(\frac{\partial g_{a1}}{\partial f} + \frac{\partial g_{11}}{\partial \theta_a} - \frac{\partial g_{a1}}{\partial f}\right) \\
 &=& 0
 \end{eqnarray*} 
since $g_{a1} \equiv 0$ and $\nabla_a g_{11} = -\frac{\nabla_a G}{G^2} = 0$.
This implies
\begin{equation}
\label{eq-hess-R}
\nabla_a\nabla_1 R = \frac{\partial}{\partial f}\left(\frac{\partial R}{\partial\theta_a}\right) = 0.
\end{equation}
On the other hand, by (\ref{eq-yam-id}) we have
$$(n-1)|\nabla f|\nabla_1 R = R_{11}.$$
Differentiating this equality  in the direction of the vector $e_a$ and  using that $\nabla_a G = 0$, where $G = |\nabla f|^2$, yields to
$$(n-1)|\nabla f|\nabla_a\nabla_1 R = \nabla_a R_{11}.$$ 
Using (\ref{eq-hess-R}) we conclude that 
$$\nabla_a R_{11} = 0$$
that is, $R_{11} = \lambda$ is constant on $\Sigma_c$. Since $R$ and $R_{11}$ are constant on $\Sigma_c$, by (\ref{eq-sect}) it follows that $R_{abab}$ is constant on $\Sigma_c$.
Since  $H$ is also constant on $\Sigma_c$ by part (ii),  (\ref{eq-sect-surf})  immediately implies  that the sectional curvatures of $\Sigma_c$ are constant, which proves (iii).
\end{proof}

\medskip

{\em Yamabe Shrinkers and Expanders:}
We will indicate how one argues in the case of shrinkers and expanders that satisfy 
(\ref{eqn-ys}), for $\rho = 1$ and $\rho = -1$, respectively.
First of all, the same arguments as before yield to
$$\nabla G = 2R \, \nabla f, \qquad (n-1)\nabla R = \ric \, (\nabla f, \cdot)$$
with $G = |\nabla f|^2$.

To prove Proposition \ref{prop-ricci} for shrinkers and expanders we can proceed with exactly the same reasoning and calculation. In other words, we still define 
$$Z(g,X) = (n -1)\Delta R + \langle \nabla R, X\rangle + \frac{1}{2(n-1)}R_{ij} X^iX^j + R^2$$
and choose $X$ to be the vector field such that 
$$\nabla_i R + \frac{1}{n-1}R_{ij} X^j = 0.$$
In the case of  Yamabe shrinkers ($\rho = 1$) and Yamabe expanders ($\rho = -1$), satisfying $(R-\rho)\, g_{ij} = \nabla_i\nabla_j f$,  if $X = \nabla f$,
assuming that they become extinct at $T=0$,  we get
$$Z(g,X) = \frac{\rho}{(-t)}\,  R$$
and therefore, 
$$\frac{\partial Z}{\partial t} = \frac{\rho}{t^2}\,R - \frac{\rho}{t}\,((n-1)\Delta R + R^2),$$
If we plug all that in (\ref{eq-evolution-Z}), using (\ref{eqn-ys}) for $\rho = 1$,  we obtain
$$\frac{\rho}{t^2}\, R  - \frac{\rho}{t}\, R^2 = -\frac{\rho}{t}\,  R^2 + A_{ij}\nabla_i f\nabla_j f + \frac{\rho}{t^2}\,R$$
implying that
\begin{equation}
\label{eq-A10}
A_{ij}\nabla_i\nabla_j f = 0.
\end{equation}

In the case of  expanders ($\rho = -1$) we argue exactly the same  way as before. Since $(R+1)g_{ij} = \nabla_i\nabla_j f$, if we consider the ones with positive sectional curvature, $f$ is still strictly convex and has at most one critical point.  In the case of  Yamabe shrinkers ($\rho = -1$), that is,  $(R-1)g_{ij} = \nabla_i\nabla_j f$, even though $R > 0$, $f$ may not be convex so we need to argue slightly differently, as in \cite{CC}. Note that the set $\{q\,\, |\,\, \nabla f(q) = 0\}$ is of measure zero. The same argument as above, for  steady solitons, gives us that locally, our soliton is rotationally symmetric. In other words, whenever $|\nabla f|(p) \neq 0$ we prove rotational symmetry in the neighborhood of the level surface $\Sigma_{f(p)}$. This means that locally, our soliton has a warped product structure 
\begin{equation}
\label{eq-warping}
g = ds^2 + \psi^2(s) g_{S^{n-1}}.
\end{equation}
Look at a cross section $S^{n-1}$ of our manifold at a point $p$, in which neighborhood the manifold is rotationally symmetric and we have the warped product structure. Assume that cross section corresponds to $s = 0$. Then $s$ measures the distance from the cross section on both  sides from it and our metric is of form (\ref{eq-warping}) for $s\in (-a,b)$, for $a, b > 0$.  As long as the warping function is not zero we can extend the warping product structure. In other words, if $\psi(s_0) \neq 0$, then by the continuity the metric will have the warping product structure $ds^2 + \psi^2(s,\theta) g_{S^{n-1}}$ a little bit past  $s_0$. Since the set of critical points of $f$ is of measure zero, by using the arguments as above to prove the rotational symmetry in the neighborhood of level surfaces corresponding to regular values, we get that $\psi(s,\theta)$ is almost everywhere the function of $s$ only. Therefore by the smoothness of  the metric, $g$ has to be of the form (\ref{eq-warping}) everywhere as long as $\psi$ does not vanish. We can have three possible scenarios: 
\begin{enumerate}
\item[(i)]
$g$ has the form (\ref{eq-warping}) for all $s\in (-\infty,\infty)$ in which case our soliton splits off a line, and that contradicts the positivity of curvature. 
\item[(ii)]
$g$ has the form (\ref{eq-warping}) for all $s\in (-\infty, a)$ and $\psi(a) = 0$, or for all $s\in (-b,\infty)$ and $\psi(-b) = 0$, which corresponds to soliton having only one end and $f$ having exactly one critical point.
\item[(iii)]
$g$ has the form (\ref{eq-warping}) for all $s\in (-a,b)$ and $\psi(-a) = \psi(b) = 0$, which corresponds to having a compact Yamabe soliton and these have been discussed and classified in Proposition \ref{prop-compact}.
\end{enumerate}

\end{proof}

We will now give the proof of Proposition \ref{prop-compact}, where no geometric assumptions have been imposed.

\begin{proof}[Proof of Proposition \ref{prop-compact}]
Tracing the soliton equation yields
$$\Delta f = n \, (R - \rho).$$
Furthermore,
\begin{equation*}
\begin{split}
n^2\int_M (R - \rho)^2\, dV_g &= \int_M (\Delta f)^2\, dV_g = \\
&= \int_M \nabla_i\nabla_i f \cdot \nabla_j\nabla_j f \, dV_g = -\int_M \nabla_i f \cdot\nabla_i\nabla_j\nabla_j f.
\end{split}
\end{equation*}
Using the identity 
$$\nabla_i\nabla_j\nabla_j f = \nabla_j\nabla_i\nabla_j f - R_{ik}\nabla_k f,$$
we obtain, integrating by parts once again, that 
$$\int_M (\Delta f)^2\, dV_g = \int_M |\nabla^2 f|^2\, dV_g + \int_M \ric(\nabla f, \nabla f)\, dV_g.$$
By (\ref{eq-yam-id}), using the soliton equation and the trace of it over and over again,
\begin{eqnarray*}
n^2\int_M (R - \rho)^2\, dV_g &=& \int_M |\nabla^2 f|^2\, dV_g + (n-1) \int_M \langle \nabla R, \nabla f\rangle\, dV_g \\
&=& \int_M |\nabla^2 f|^2\, dV_g  - (n-1)\int_M (R-\rho)\cdot\nabla f \, dV_g \\
&=& \int_M |\nabla^2 f|^2\, dV_g  - n(n-1)\int_M (R-\rho)^2\, dV_g \\
&=& n\int_M (R-\rho)^2\, dV_g - n(n-1)\int_M (R-\rho)^2\, dV_g 
\end{eqnarray*}
implying that $R \equiv \rho$, finishing the proof of the Proposition.

\end{proof}

\begin{prop}
\label{prop-conf-flat}
All complete, noncompact rotationally symmetric   steady or expanding Yamabe solitons with positive  Ricci curvature are either  nonflat and globally conformally equivalent to $\mathbb{R}^n$, or flat. In the case of shrinkers, they are either flat, or locally isometric to cylinders or nonflat and globally conformally equivalent to $\mathbb{R}^n$.
\end{prop}

\begin{proof}
It is known that every rotationally symmetric metric is locally conformally flat.
In \cite{CH} it has been showed that all complete, locally conformally flat manifolds of dimension $n \ge 3$ with nonnegative Ricci curvature enjoy nice rigidity properties: they are either flat, or locally isometric to a product of a sphere and a line, or are globally conformally equivalent to $R^n$ or to a spherical spaceform $S^n/\Gamma$.  The second case contradicts our assumption on positive curvature. Hence, the only possibility for complete, nonflat, locally conformally flat, steady Yamabe solitons is being globally conformally equivalent to the euclidean space.
\end{proof}

Proposition \ref{prop-conf-flat} and Theorem \ref{thm-unique} imply that for classifying gradient Yamabe solitons with positive sectional curvatures is enough to understand and classify the radial Yamabe soliton solutions of the form $g = u^{\frac{4}{n+2}} dx^2$. We will study rotationally symmetric, conformally flat gradient Yamabe solitons and their geometric properties in the next couple of sections.

\section{PDE formulation of Yamabe solitons}
\label{cf}

Our aim in this section is to prove Proposition \ref{prop-ell00}.  
We will assume that  the metric $g$ is globally conformally equivalent to $\R^n$ (we will call it conformally flat) and rotationally symmetric and that satisfies
\eqref{eqn-ys}. We may   express $g$  as  $$g = u(r)^{\frac{4}{n+2}}\, (dr^2 + r^2 g_{S^{n-1}})$$
where 
$(r, \theta_2, \dots, \theta_n)$  denote  spherical coordinates.

We choose next   cylindrical coordinates on $\mathbb{R}^n$ defining  $v(s)$ by 
\begin{equation}\label{eqn-cyl}
v(s)^{\frac{4}{n+2}} =r^2 u(r)^{\frac{4}{n+2}}, \qquad  r = e^s.
\end{equation}
Then
$g = v(s)^{\frac{4}{n+2}}\, ds_c^2$, 
where $ds_c = ds^2 + g_{S^{n-1}}$ is the cylindrical metric. Denote by $w$ the conformal
factor in cylindrical coordinates, namely
$w(s) = v(s)^{\frac{4}{n+2}}.$

We will use an index $1$ or $s$ to refer to the $s$ direction and indices $2, 3, \dots, n$ to refer to the spherical directions.  By \eqref{eqn-ys} we have 
\begin{equation}
\label{eq-sol12}
(R-\rho)  \, g_{ij} = \nabla_i\nabla_j f
\end{equation}
for a potential function $f$ which is  radially symmetric.
Using the formulas 
$$\nabla_i\nabla_j f = f_{ij} + \Gamma_{ij}^l \, f_l$$
and
$$\Gamma_{ij}^l = \frac{g^{kl}}2 \, \left ( \frac{\partial g_{ik}}{\partial{x_j}} +  \frac{\partial g_{jk}}{\partial{x_i}} -  \frac{\partial g_{ij}}{\partial{x_k}}  \right )$$
for a function $f=f(s)$ that only depends on $s$ we have 
$$\nabla_s\nabla_s f = f_{ss} - \Gamma_{ss}^s f_s \qquad \mbox{and} \qquad 
\nabla_i\nabla_i f = -\Gamma_{ii}^s f_s.$$
Since
$$\Gamma_{ss}^s = \frac{w_s}{2w}, \qquad \Gamma_{11}^s = - \frac{w_s}{2w}$$
we conclude that 
$$\nabla_s\nabla_s f =  f_{ss} - \frac{w_s\, f_s}{2w} 
\qquad \mbox{and} \qquad \nabla_1\nabla_1 f = \frac{w_s\,  f_s}{2w}.$$
The last two relations and the soliton equation (\ref{eq-sol12}) imply
\begin{equation}
\label{eq-relations}
f_{ss} - \frac{w_s\,  f_s}{2w} = (R-\rho) \, w \qquad \mbox{and} \qquad \frac{w_s\,  f_s}{2w} = (R-\rho)\, w.
\end{equation}
If we subtract the second equation from the first we get
$$f_{ss} - \frac{f_s w_s}{w} = 0.$$
This is equivalent to 
$\left(\frac{f_s}{w}\right)_s = 0$ (since $w > 0$)
which implies that 
\begin{equation}
\label{eq-nice}
\frac{f_s}{w} = C. 
\end{equation}
The scalar curvature $R$ of the metric  $g = w(s) \, (ds^2 + g_{S^{n-1}})$  is given by
\begin{equation}\label{eqn-R}
R = - \frac{4(n-1)}{n-2}\, w^{-\frac{n+2}4} \left (\, (w^{\frac{n-2}{4}})_{ss} - \frac{(n-2)^2}{4} w^{\frac{n-2}{4}} \right ). 
\end{equation}
The second equation in   (\ref{eq-relations}) and (\ref{eq-nice})  imply that
\begin{equation}\label{eqn-Rw}
w_s = \frac 2C \, ( R -\rho)\, w.
\end{equation}
Combining \eqref{eqn-R} and \eqref{eqn-Rw}  gives
$$
w_s = -   \frac{8(n-1)}{C(n-2)}\, w^{-\frac{n+2}4+1}  \left (\, (w^{\frac{n-2}{4}})_{ss} - \frac{(n-2)^2}{4} w^{\frac{n-2}{4}}  + \rho \, \frac{(n-2)}{4(n-1)} \,w^{\frac{n+2}4}  \right ). 
$$
Setting $\theta = \frac{C\, m}{2(n-1)}$ we conclude that  $w$ satisfies the equation
\begin{equation}\label{eqn-ww}
(w^{\frac{n-2}{4}})_{ss} + \theta  \,  (w^{\frac{n+2}4})_s   - \frac{(n-2)^2}{4} w^{\frac{n-2}{4}}  + \rho \, \frac{(n-2)}{4(n-1)}
 \,w^{\frac{n+2}4}=0. 
\end{equation}
To facilitate future references, we also remark  that  \eqref{eqn-ww} can be re-written as
\begin{equation}
\label{eq-ww2}
w_{ss} = \frac{(\alpha - 1)}{\alpha} \, \frac{w_s^2}{w} - (\alpha+1)\, \theta \, w_s w + \frac{4}{\alpha} \, w - \frac{\alpha+4}{\alpha}
 \, \rho \, w^2, 
\quad \alpha=\frac{4}{n-2}.
\end{equation}
We conclude from \eqref{eqn-ww}  that $g = v(s)^{\frac{4}{n+2}}\, ds_c^2 = w(s) \, \, ds_c^2$ 
is a Yamabe soliton   if and only if $v$ satisfies 
the  equation 
\begin{equation}
\label{eqn-vvv}
(v^{\frac {n-2}{n+2}})_{ss} + \theta  \, v_s - \frac{(n-2)^2}{4} v^{\frac {n-2}{n+2}} + \rho \, \frac{(n-2)}{4(n-1)}\, v =0.
\end{equation}
If we go back to Euclidean coordinates, i.e. we set
\begin{equation}\label{eqn-coor}
u^{\frac 4{n+2}}(r)   = e^{ - 2 \, s} v^{\frac 4{n+2}} (s), \qquad s=\log r
\end{equation}
then, after a direct calculation, we conclude that 
$u$ satisfies the elliptic equation
\begin{equation}\label{eqn-ell2222}
\Delta u^m  + \theta     \, x \cdot \nabla u  +  \frac 1{1-m} \, (2\theta  + \frac m{n-1}\, \rho)\,  u =0, \qquad
 m=\frac{n-2}{n+2}
\end{equation}
which can also be written as
\begin{equation}\label{eqn-ell2}
\frac{n-1}{m} \, \Delta u^m  + \beta     \, x \cdot \nabla u  +  \gamma \,  u =0
\end{equation} 
with
$$\beta = \frac{n-1}{m}\, \theta \quad \mbox{and} \quad  \gamma = \frac{2\beta +\rho}{1-m}.$$

Observe also, that if $g=u^{\frac 4{n+2}}\, dx^2$ is a radially symmetric smooth solution of 
equation \eqref{eqn-ell2}, then the above discussion (done backwards) implies that $g$
satisfies the Yamabe solition equation  \eqref{eqn-ys} with potential function $f$
defined in terms of $w$ by \eqref{eq-nice}. 

To finish the proof of Proposition \ref{prop-ell00}, we need to show the following:

\begin{claim}\label{claim-beta} If $g=u^{\frac 4{n+2}}\, dx^2$ defines a complete Yamabe gradient soliton,  then 
$\beta \geq  0$. In the case of a  noncompact Yamabe shrinker or expander, then $\beta >0$.  

\end{claim}

\begin{proof} 
We have seen in Remark \ref{rem-par2} that a Yamabe soliton $g=u^{\frac 4{n+2}}\, dx^2$
defines a solution $\bar g=\bar u^{\frac 4{n+2}}\, dx^2$ of the Yamabe flow \eqref{eqn-yf},
or equivalently, of  the fast diffusion equation \eqref{eqn-fde}. Hence, if  $g$ is a Yamabe shrinker or a steady
soliton, then the scalar curvature $\bar R$ of $\bar g$ satisfies $\bar R \geq 0$ (this can be seen by using Aronson-Benilan inequality). It follows that  
the scalar curvature $R= - \frac{4\, (n-1)}{n-2} \,  {\Delta u^m}/u$  
  of the metric $g$ satisfies $R \geq 0$ as well. 
Equation 
\eqref{eqn-ell2} implies that 
\begin{equation}\label{eqn-RRR}
R = (1-m) \, (\gamma +\beta r\, (\log u)_r )
=(2\beta +\rho) + (1-m)\beta r\, (\log u)_r
\end{equation}
since $1-m=4/(n+2)$ and $\gamma\, (1-m)=2\beta +\rho$. 
Hence, 
$R(0) = (1-m) \, \gamma = 2\beta +\rho$. We conclude that  $\gamma \geq 0$ on a Yamabe shrinker or
a steady soliton (since $\rho=1$ or  $\rho=0$ respectively). 

In the case of a Yamabe shrinker, it is shown in Proposition 7.4 in \cite{Vaz1} that 
$\gamma >  \frac 1{1-m}$, or equivalently  $\beta >  0$,  (otherwise \eqref{eqn-ell2} does not admit a smooth global solution). This, in particular,  implies that on a Yamabe shrinker $R(0) > 1$. 

In the case of a Yamabe steady soliton, $\gamma \geq 0$ implies that $\beta \geq 0$ as well.
If $\gamma =0$, then $\beta =0$ and it follows from \eqref{eq-nice} that $w$ and hence $u$
is constant (remember that $C=\frac{2(n-1)}{m}\, \theta$ in \eqref{eq-nice})
and $\theta = \frac{m}{n-1}\beta$). In this case, $u$ defines the flat metric.

It remains the prove the claim for the Yamabe expanders.  To this end, we observe first that if there were 
a smooth solution of \eqref{eqn-ell2} with $\beta \leq 0$ (which implies that $\gamma \leq -\frac \rho 2 <0$ as well)
then, $\bar u(x,t) := t^{-\gamma} u(x\, t^{-\beta})$ would be a solution of \eqref{eqn-fde}
with initial data identically equal to zero. The uniqueness result in \cite{HP} would imply
that $\bar u \equiv 0$. Hence, $\beta > 0$. 

\end{proof}
\section{Classification of  radially symmetric Yamabe solitons}

In this section we will discuss the existence of radially symmetric  and conformally flat  Yamabe solitons.  We have seen in the previous section that this is equivalent to having 
a global solution  of equation \eqref{eqn-ell00}. We will then discuss the proof of Proposition \ref{prop-cys}.
Before we proceed with its proof we give the following a priori bound on the scalar curvature $R$.

\begin{prop}\label{prop-scalar} Let $g=u(r)^{\frac 4{n+2}} dx^2$ be a rotationally symmetric Yamabe soliton
with scalar curvature $R$. 
(i) If $g$ is a  Yamabe shrinker, then    $R >1$  
as long as  $\beta > \frac{1}{n-2}$. (ii) If $g$ is a  Yamabe 
steady solition or a Yamabe expander, then    $R >0$ 
as long as  $\gamma >0$. (iii) If $g$ is a   Yamabe expander, then  $R <0$  
as long as  $\gamma <0$ and $\beta >0$.

\end{prop}

\begin{proof} Let $g=u^{\frac 4{n+2}} dx^2$ 
be  a Yamabe soliton which satisfies \eqref{eqn-ell2} and set $v=u^\frac 4{n+2}$. It easily follows (\cite{Ch}) that the scalar curvature of a Yamabe soliton (\ref{eqn-ys}) satisfies the following elliptic equation
$$(n-1)\Delta_g R + \frac{1}{2}\langle \nabla R, \nabla f\rangle_g + R(R-\rho) = 0.$$
In the case of a rotationally symmetric Yamabe soliton we have showed that $f_s = C w$, with $C = 2\beta$.
All these yield to
\begin{equation}\label{eqn-RR4}
(n-1)\, \Delta R + \beta \, x \cdot \nabla R \, v + R \, (R-\rho)\, v =0, \qquad \rho \in \{0, +1,-1 \},
\end{equation}
where the Laplacian and the gradient are taken with respect to the usual euclidean metric.

Assume first that $g$ is a Yamabe shrinker so that $\rho =1$ in \eqref{eqn-RR4} and set $ \bar R = R-1$
which satisfies
\begin{equation}\label{eqn-RR5}
(n-1)\, \Delta \bar R + (1-m) \, \beta \, x \cdot \nabla \bar R \, v + \bar R \, (\bar R+1)\, v =0.\end{equation}
Since $\bar R$ is a radial function, integrating \eqref{eqn-RR5} in a ball $B_r:=B_r(0)$  we obtain (after integration by parts)
\begin{equation}\label{eqn-RR6}
\begin{split}
 ((n-1) \bar R_r &+  \beta r \bar R \, v) \, |\partial B_r|  \\&=  \beta  \int_{B_r}  r\, v_r \bar R \, dx  +   \int_{B_r}  \bar R \, [ \, n   \beta -1 -\bar R \, ]\, v \, dx.
 \end{split}
 \end{equation}
Equation \eqref{eqn-RRR} and equalities $v=u^{1-m}$ and $\gamma (1-m) = 2 \beta +1$ yield to
$$  \bar R = R-1 = 2\beta + \beta \, r \frac {v_r}{v}$$
which implies that $\beta r v_r = \bar R \, v - 2\beta v.$
Substituting  this into \eqref{eqn-RR6} gives 
\begin{equation}\label{eqn-RR7}
 ((n-1) \bar R_r +  \beta r \bar R \, v) \, |\partial B_r|   = [\, (n-2)  \beta  -1 ] \,   \int_{B_r}  \bar R \,  v \, dx.
\end{equation}
Since  $(n-2)\beta > 1$ we conclude that 
\begin{equation}\label{eqn-RR77}
(n-1) \bar R_r + \beta r \bar R \, v >0.
\end{equation}
We will now  show that $\bar R >0$. From \eqref{eqn-RRR} we have that $R(0)=2\beta +1 >1$,
since $\beta >0$. Hence $\bar R >0$ near $r=0$. Equation \eqref{eqn-RR77} now readily implies that $\bar R$
remains positive if $ \beta > 1/(n-2)$. 

Assume next that $g$ is a Yamabe expander so that $\rho =-1$ in \eqref{eqn-RR4}. Integrating equation
\eqref{eqn-RR5} as before,  we obtain that \eqref{eqn-RR6} holds for $R$ instead of $\bar R$. 
We recall that this time $(1-m)\, \gamma = 2\beta -1$. Hence, by \eqref{eqn-RRR} we have 
$\beta r v_r = (R+1-2\beta)\, v$. Substituting  this into \eqref{eqn-RR6} (for $R$ instead of $\bar R$) yields
\begin{equation}\label{eqn-RR8}
\begin{split}
 ((n-1) R_r +  \beta r R \, v) \, |\partial B_r|  =   
 \int_{B_r}   (n-2) \beta \,  R\, v \, dx.
 \end{split}
 \end{equation}
If  $R(0) = \gamma (1-m) =2\beta -1>0$, then $R >0$ for $r$ sufficiently close to the origin.
It follows from   \eqref{eqn-RR8}
 that $R$ remains positive. If $R(0) = \gamma (1-m) <0$ and $\beta >0$, then $R <0$ for $r$ sufficiently close to the origin
 and \eqref{eqn-RR8} implies that $R$ remains negative.

On a Yamabe steady soliton we always have that $R \geq 0$. We remark that the above argument shows that
$R >0$ if $R(0) = (1-m)\gamma >0$ (which also follows from the strong maximum principle). 

\end{proof}

\begin{cor}\label{cor-RRR} Assume that $g= u(r)^{\frac 4{n+2}}  \, (dr^2 + r^2 g_{S^{n-1}})$ is a Yamabe 
 soliton  which is radially symmetric. 
Assume that  $ \beta > \frac 1{n-2}$ in the case of a Yamabe shrinker or $\gamma >0$ in the case of 
a Yamabe steady soliton or a Yamabe expander. Then $R$ is a decreasing function in $r$. 
\end{cor}

\begin{proof}  The function $R$ satisfies the elliptic equation \eqref{eqn-RR4} and by our assumptions 
and Proposition \ref{prop-scalar}, we have that $R\, (R - \rho) >0$ everywhere. 
Since $u$ is  strictly positive  equation  \eqref{eqn-RR4} implies that $R$ cannot achieve a local
minimum at a point $x \in \R^n$. Since $R$ is a radial function and $R(0) >0$  it follows that $R$
must be a decreasing function of $R$.

\end{proof}

\begin{proof}[Proof of Proposition \ref{prop-cys}] We will separate the cases $\rho =0$ 
(steady solitions), $\rho = 1$ (shrinkers) and  $\rho = -1$ (expanders).  

{\em Yamabe shrinkers $\rho =1$:} In this case the result is proven in Proposition 7.4 in \cite{Vaz1}
(see also \cite{GP}). We only need to  remark that $u$ solves \eqref{eqn-ell00} if and only if
$\bar u = (T-t)^\gamma \, u(x\, (T-t)^\beta)$ is an ancient  self-similar solution of the fast diffusion equation
\eqref{eqn-fde}.

\smallskip

{\em Yamabe expanders $\rho =-1$:} We look for a smooth  global radially symmetric solution of the elliptic equation 
\begin{equation}
\label{eqn-exp}
\frac{n-1}m\, \Delta u^m + \beta    \, x \cdot \nabla u  +  \gamma \,  u =0, \qquad \mbox{on}\,\, \R^n
\end{equation}
with $\beta >0$ and $\gamma = \frac{2\beta-1}{1-m}$, $m=\frac{n-2}{n+2}$.  

It is  well known     (c.f. in \cite{Vaz1}, Section 3.2.2)   that for any $\lambda >0$, equation \eqref{eqn-exp} 
admits a unique smooth radial solution $u=u_\lambda(r)$, with $u_\lambda (0) =\lambda$ 
and which is defined in a neighborhood of the origin. This follows via the change of variables
$$r=e^s, \qquad X(s)=\frac{r\, u_r}{u}, \qquad Y(r)=r^2\, u^{1-m}$$
which (in the radial case) transforms equation \eqref{eqn-exp} to an autonomous system for $X$ and $Y$. 

Hence we only need to show that such solution is globally defined. To this end it suffices to prove that $u$ remains positive
and bounded.  

We will first show that the $u$ remains positive for all $r >0$. This simply follows from
expressing \eqref{eqn-exp}  as
$$\mbox{div} \left (\frac{n-1}m \,  \nabla u^m  + \beta \, x \cdot  u \right )  = ( n \beta - \gamma) \, u $$
and observing that $( n \beta - \gamma) \, u \geq 0$,  in our case,  as long as $u \geq 0$. 
Integrating in a ball $B_r(0)$ gives  the differential inequality 
$\frac {n-1}{m}\, (u^m)_r + \beta  r \, u  \geq 0$, 
which easily  implies  the lower bound
$
u(r) \geq \left (  \mu + \frac{\beta \, (1-m)}{2 (n-1)} r^2    \right )^{-\frac 1{1-m}}$, 
with $\mu = u(0)^{m-1}.$ 

To establish the bound from above we argue as follows. If $\gamma >0$, then  $R >0$ by  Proposition \ref{prop-scalar}. Hence $\Delta u^m \leq 0$, which gives  $(r^{n-1} (u^m)_r)_r \le 0$, yielding to $u_r \le 0$ and therefore giving the upper bound on $u$. If $\gamma < 0$, then  $R <0$ by  Proposition \ref{prop-scalar}
By \eqref{eqn-RRR} we obtain  the inequality 
$(\log u)_r \le -\frac{\gamma}{\beta r}$ which implies the bound from above
$u(r) \le C \, r^{-\gamma/\beta}$.


\smallskip

{\em Yamabe steady solitions $\rho =0$:} We will show, for any given $\beta >0$,  the existence of an one parameter family 
of radial solutions $u_\lambda$, $\lambda >0$ of equation
\begin{equation}\label{eqn-ells}
\frac {n-1}{m} \, \Delta u^m + \beta  \, x \cdot \nabla u  +  \gamma  \, u =0
\quad \mbox{on}\,\,\, \R^n, \quad \gamma = \frac{2 \beta}{1-m}. 
\end{equation}
Notice, that $u$ solves \eqref{eqn-ells}
 if and only if
$\bar u = e^{-\gamma t} \, u(x\, e^{-\beta t})$ is an eternal  self-similar solution of the fast diffusion equation
\eqref{eqn-fde}. The existence of such solutions $\bar u$ (without a proof) is  noted in \cite{GP}. 
We only outline the proof,  avoiding the details  of standard well known  arguments. 

It follows from standard ODE arguments that for any $\lambda >0$, equation \eqref{eqn-ells} 
admits a unique smooth radial solution $u_\lambda$, with $u_\lambda (0) =\lambda$ 
and which is defined in a neighborhood of the origin. Hence we only need to show that such solution
is globally defined and satisfies the asymptotic behavior $u(r) \approx ( \frac{\log r}{r^2})^{\frac 1{1-m}}$,  as $r \to \infty$.

To this end, it is more convenient to work in cylindrical coordinates. Recall that if $v(s) = r^{\frac 2{1-m}}\, u(r)$,
$s=\log r$, then $v$ satisfies equation \eqref{eqn-vvv} with $\rho=0$ and $\theta >0$, namely   
\begin{equation}\label{eqn-vv}
 (v^m)_{ss} + \theta  \, v_s - \frac{(n-2)^2}{4} v^{m}  =0, \qquad m=\frac{n-2}{n+2}.
 \end{equation}
We assume that $v$ is defined on $-\infty < s \leq s_0$, for some $s_0 \in \R$. We will first observe that
$v_s \geq 0$ for all $s \leq s_0$. Indeed, since $v(-\infty)=0$ and $v >0$, it follows that $v_s >0$
near $-\infty$. To show that the inequality is preserved we argue that near a point  $s_1 < s_0$
at which $v_s(s_1)=0$, we  still have $v(s_1) >0$, hence by the above equation $(v^m)_{ss} >0$,
which implies that $(v^m)_s$ has to increase and hence it cannot vanish. 

Once we know that $v_s \geq 0$, equation \eqref{eqn-vv} implies that 
$(v^m)_{ss} \leq C v^{m}$, $C=\frac{(n-2)^2}{4}.$
Setting $h=(v^m)_s$ and considering $h$ as a function of $z=v^m$ we find that $h$ satisfies 
$h  h' \leq  C \, z$, or equivalently $f=h^2$ satisfies $f' \leq 2C\, z$. Since $f(0)=0$ (this corresponds
to $s=-\infty$) we conclude that  $f(z) \leq C\, z^2$, or 
$(v^m)_s \leq C \, v^m$, which readily implies that $v^m$ will remain bounded for all $s \in \R$. This
proves that for each $\lambda$  
the solution $u_\lambda$ is globally defined on $\R^n$.

It remains to show that $u(r)^{1-m} \approx r^{-2} \, \log r$,  as $r \to \infty$. This will be shown
separately in what follows.

\end{proof}

\begin{prop}\label{prop-asympt}
If $g_\mu= u_\mu^{\frac 4{n+2}}(r)  \, dx^2$
is a radially symmetric  steady  soliton as in Proposition \ref{prop-cys}, then 
$$u_\mu^{\frac 4{n+2}} (r) \approx  \frac{\log r}{r^2}, \qquad \mbox{as}\,\,\, r \to \infty.$$
\end{prop}

\begin{proof}
We will use cylindrical coordinates and  show that if $g_\mu= w(s) \, ds_c^2$ is a non-trivial steady Yamabe soliton,  then  $w(s)= O(s)$ as $s \to \infty$.
More precisely, we will show there exist constants $c, C > 0$ so that
\begin{equation}
\label{eq-bound-want}
c\, s \le w(s) \le C\, s, \qquad \mbox{as} \,\,\, s\to\infty.
\end{equation}
Recall that $w$ satisfies the equation  \eqref{eq-ww2} with $\rho=0$, namely
\begin{equation}
\label{eq-ww3}
w_{ss} = \frac{(\alpha - 1)}{\alpha} \, \frac{w_s^2}{w} - (\alpha+1)\, \theta \, w_s w + \frac{4}{\alpha} \, w, 
\quad \alpha=\frac{4}{n-2}.
\end{equation}

Assume first $\alpha:=4/(n-2) < 1$. We will first show the bound from above in (\ref{eq-bound-want}). It follows from
\eqref{eq-ww2} that
$$w_{ss} \le  (\alpha +1) \theta\,  w\, \left(\frac{4}{\alpha (\alpha +1) \theta} - w_s\right).$$
Assume there exists $s_0$ so that $w_s(s_0) \ge \frac{4}{\alpha (\alpha +1)\theta}$, since otherwise we are done. The above inequality then implies
$$w_s \le \max\left\{\frac{4}{\alpha (\alpha +1) \theta}, w_s(s_0) \right\}$$
giving us the upper bound in (\ref{eq-bound-want}).
For the lower bound, first observe that the $\lim_{s\to\infty} w = \infty$ 
(since solutions of the Yamabe flow with $w \leq C$, or equivalently $u \leq C\, r^{- \frac 2{1-m}}$,  vanish in finite time
which is impossible for a steady soliton). 
 If for some big $s$ the right hand side in (\ref{eq-ww2}) is negative, that is,
$$w_s^2 + \frac{\theta  \alpha (\alpha +1)}{1-\alpha} w_s w^2 - \frac{4}{\alpha(1-\alpha)} w^2 \ge 0$$
then setting  $A= \frac{\theta  \alpha (\alpha +1)}{1-\alpha}$ and $B=\frac{4}{\alpha(1-\alpha)}$,  we have (since $w_s \geq 0$) that 
$$2\, w_s \ge -A \,  w^2 + w^2 \sqrt{ A^2    + \frac {4 B}{ w^{2}}} \geq B\, w $$
when $w$ is very large (which is true always when $s\to\infty$). This is  impossible since
 we have just shown that $w_s$  remains bounded, as $s \to \infty$. We conclude that 
the right hand side in (\ref{eq-ww2}) is always positive that means $w_s$ is increasing. 

The case $\alpha > 1$  can be treated similarly as above. The case $\alpha =1$ is simpler. 
\end{proof}

We conclude this section by  showing the positivity of the sectional curvatures of the Yamabe 
solitons  found in Proposition \ref{prop-cys} in most of the cases.

\begin{prop}\label{prop-sec}
The logarithmic cigars and the  Yamabe expanders   found in Proposition \ref{prop-cys} 
have strictly  positive sectional curvatures as long as $\gamma >0$. The Yamabe shrinkers 
have strictly  positive sectional curvatures as long as $\beta > \frac{1}{n-2}$. 
\end{prop}

\begin{proof}
Recall that if $w(s)$ is the conformal factor in cylindrical coordinates, we found in (\ref{eqn-Rw}) that
$$w_s = \frac{1}{\beta} (R - \rho) w$$
where $\rho = 0$ for the steady solitons and $\rho = 1$ for the shrinkers and $\rho=-1$ for the expanders.  The above equality and  Proposition \ref{prop-cys} imply that  $w_s >0$,
assuming that $\gamma >0$ always and that $\beta > \frac{1}{n-2}$ in the case of a Yamabe
shrinker.  Also, by Corollary  \ref{cor-RRR}, $R_s \leq 0$. We will show that
the last two inequalities imply that the sectional curvatures are nonnegative. 

To express the sectional curvatures in terms of $w$, we consider the   
the geodesic distance $\tilde{s}$ from the origin, that is,
$$d\tilde{s} = \sqrt{w(s)}\, ds \qquad \mbox{or} \qquad  \tilde{s}(s) = \int_{-\infty}^s \sqrt{w(u)}\, du.$$
Then our metric reads as
$g = d\tilde{s}^2 + \psi^2(\tilde{s}) \, g_{S^{n-1}},$
with $\psi^2(\tilde{s}) = w(s)$ and $\tilde{s} \in [0,\infty)$. Note that we have $\psi(0) = 0$. 
Differentiating  $\psi(\tilde{s})^2 = w(s)$ in $s$ yields to $$2\psi \psi_{\tilde{s}}\,  \sqrt{w} = w_s.$$ Since   $w_s > 0$, we have $\psi_{\tilde{s}} > 0$. 

Denote by  $K_0$ and $K_1$   the sectional curvatures of the $2$-planes perpendicular to the spheres $\{x\}\times S^{n-1}$ and the $2$-planes tangential to these spheres, respectively. These curvatures are given by 
$$K_0 = -\frac{\psi_{\tilde{s}\tilde{s}}}{\psi}  \qquad \mbox{and} \qquad  K_1 = \frac{1 - \psi_{\tilde{s}}^2}{\psi^2}.$$

We will first show that $K_0 \geq 0$, namely  that $- \psi_{\tilde{s}\tilde{s}} \geq 0$.
By direct calculation this is equivalent to $ - (\log w)_{ss} >  0$. By (\ref{eqn-Rw}), the last inequality is equivalent to 
$R_s \leq  0$ 
which follows from Corollary \ref{cor-RRR}.

We will now show that $K_1 \geq 0$. The inequality  $K_0 \geq 0$ implies  that $\psi_{\tilde{s}\tilde{s}} \le 0$. By Proposition 4.1. in \cite{AK}  and $\psi_{\tilde{s}} > 0$ we get
$$\lim_{\tilde{s}\to 0}\psi_{\tilde{s}} = 1.$$
Since $\psi_{\tilde{s}\tilde{s}} \le 0$, implying that $\psi_{\tilde{s}}$ decreases in $\tilde{s}$, we obtain that 
$$0 < \psi_{\tilde{s}} \le 1, \qquad \mbox{for all} \,\,\, \tilde{s} \in [0,\infty).$$
This shows that
$K_1 \ge 0.$

We will now show that both $K_0$ and $K_1$ are strictly positive. 

\begin{claim}
We have $K_0 >0$ and $K_1 >0$. 
\end{claim}

\begin{proof}[Proof of Claim]
We have observed above that $w_s > 0$. Recall that $w$ satisfies the equation \eqref{eq-ww2}, namely 
\begin{equation}
\label{eq-w-1000}
w w_{ss}  + \frac{n-6}{4}w_s^2 + \frac \theta m \, w^2 w_s + \frac{\rho}{n-1} w^3 - (n-2)w^2 = 0, 
\end{equation}
where $m=\frac{n-2}{n+2}$ and $\rho = 0, -1, 1$, in the case of the logarithmic cigars, expanders and shrinkers, respectively. 
We have just shown that $K_0 \ge 0$ and $K_1 \ge 0$, which  are equivalent to $w_s^2 - w w_{ss} \ge 0$ and $4w^2 - w_s^2 \ge 0$, respectively.  Assume that $K_0 = 0$ at an interior point $s$. Then, at that particular point, which is the interior minimum point for $K_0$, implying $(K_0)_s = 0$ at that point, we have
\begin{equation}\label{eqn-w16}
ww_{ss} = w_s^2, \qquad w w_{sss} = w_s w_{ss}=\frac{w_s^3}{w}. 
\end{equation}
Combining the first identity with equation \eqref{eq-w-1000}  yields to
\begin{equation}
\label{eq-int-id1}
\frac{n-2}{4} \, w_s^2 + \frac{\theta}m   w^2  w_s  + \frac{\rho}{n-1}w^3 - (n-2)w^2 = 0
\end{equation}
satisfied at the interior minimum of $K_0$. If we differentiate (\ref{eq-w-1000}) in $s$ and use \eqref{eqn-w16} 
to eliminate $w_{sss}$ and $w_{ss}$ we  obtain
\begin{equation}
\label{eq-int-id2}
\left (\frac{n-2}{4}\, w_s^2 + \frac{3 \theta}{2m} \, w^2 w_s   + \frac{3\rho}{2(n-1)} w^3 - (n-2)w^2 \right )\, w_s = 0
\end{equation}
holding at the interior minimum point for $K_0$. After dividing (\ref{eq-int-id2}) by $w_s$ and subtracting (\ref{eq-int-id1}) from (\ref{eq-int-id2}) we obtain
\begin{equation}
\label{eq-rel-100}
\frac{\rho}{n-1} w^3 + \frac{\theta}m \, w^2 w_s = 0.
\end{equation}
Since $w_s >0$ we  see that this  is impossible for $\rho \ge 0$. That shows the logarithmic cigars and the shrinkers for $n\beta \ge \gamma$ have $K_0 >0$.
 In the case of the expanders ($\rho = -1$) when $\gamma > 0$, subtracting   (\ref{eq-rel-100}) from (\ref{eq-int-id1}) 
 yields to 
$w_s=2w$, which combined with \eqref{eq-rel-100}  gives  $\theta = \frac{m}{2(n-1)}$, which is equivalent to $\beta = \frac 12$ in \eqref{eqn-ell2} and therefore $\gamma = 0$. This contradicts $\gamma > 0$.

Similarly, assume $K_1 = 0$ at some interior point $s$. Then $(K_1)_s = 0$ at that point, implying 
$$w_s = 2w, \qquad w_{ss} = 4w.$$
If we plug these back in (\ref{eq-w-1000}) we obtain
\begin{equation}
\label{eq-rel-200}
 \frac {2 \theta}m w^3 + \frac{\rho}{n-1} w^3 = 0
\end{equation}
which is impossible for $\rho \ge 0$, covering the logarithmic cigars ad the shrinkers. In the case of the expanders ($\rho = -1$), when $\gamma > 0$, identity (\ref{eq-rel-200}) implies again that 
$\theta = \frac{m}{2(n-1)}$, or $\beta = \frac 12$
and therefore $\gamma = 0$ which again contradicts $\gamma > 0$.

\end{proof}

The proof of the Proposition readily follows from the  above claim.

\end{proof}

\section{Eternal solutions to the Yamabe flow}

As a corollary of the proof of Theorem \ref{thm-unique} we have the following rigidity result for eternal solutions to the Yamabe flow, which can be viewed as the analogue of Hamilton's theorem for eternal solutions to the Ricci flow, and the proof adopts some of Hamilton's ideas to the Yamabe flow.

\begin{cor}
\label{cor-eternal}
Let $g(x,t)$ be a  complete eternal solution to the locally
conformally flat Yamabe flow on a simply connected manifold $M$, with uniformly bounded sectional curvature and
strictly positive Ricci curvature. If the  scalar curvature $R$ assumes
its maximum at an interior space-time point $P_0$,  then $g(x,t)$ is necessarily a gradient steady soliton.
\end{cor}

From (\ref{equation-vf}) we have that $X_j = -(n-1)R_{ij}^{-1}\nabla R_i$.
Since
$$R_t = (n-1)\, \Delta R + R^2$$
and since at the point $P_0=(x_0,t_0)$ where $R$ assumes its maximum, we have 
${\partial R}/{\partial t} =0$ and $\nabla_i R = 0$, we conclude that 
$$Z(g,X) = 0, \qquad \mbox{at} \,\, P_0.$$
The idea is to apply
the strong maximum principle to get that $Z \equiv 0$,  which implies
that $\nabla_i X_j = R\, g_{ij}$ (this will follow from the evolution equation for
$Z$). To simplify the notation, we define $\Box = \partial_t - (n-1)\, \Delta$. 

To finish the proof of Corollary \ref{cor-eternal} we need the
following version of the strong maximum principle. 

\begin{lem}
\label{lemma-zero-every}
If $Z(g,X) = 0$ at some point at $t = t_0$, then $Z(g,X) \equiv 0$
for all $t < t_0$.
\end{lem}

\begin{proof}
The proof is similar to the proof of Lemma $4.1$ in \cite{Ha1}. For the convenience of a reader we will include the main steps of the proof. In what follows we denote
by $\Delta$ the Laplacian with respect to the metric $g_{ij}(x,t)$.
Our Lemma will be a consequence of the usual strong maximum principle,
which assures that if we have a function $h \ge 0$ which solves
$$h_t =\Delta h$$
for $t\ge 0$ and if we have $h > 0$ at some point when $t = 0$, then
 $h > 0$ everywhere  for  $t > 0$.

Assume there is a  $t_1 < t_0$ such  that $Z(g,X) \neq 0$ at some point, at
time $t_1$. We may assume, without the loss of generality, that  $t_1 = 0$. Define
$F_0 := Z(0)$ and allow $F_0$ to evolve by the  equation
$$F_t = (n-1)\,  \Delta F.$$
From  the result of Chow we know that $F(0) \ge 0$ and therefore it will
remain so for $t \ge 0$,  by the maximum principle. Since by our
assumption,  there is a point at $t = 0$ at which $F(0) > 0$, we conclude by the
strong maximum principle that  $F > 0$ everywhere as soon as $t > 0$.  

Take
$\phi = \delta e^{At}f(x)$, where $f(x)$ is the function constructed
in \cite{Ha2} with $f(x) \to \infty$ as $x\to \infty$, $f(x) \ge 1$
everywhere, with all the covariant derivatives bounded, and $A$ is
big enough (depending on $\delta$) so that
$$\phi_t > (n-1) \, \Delta \phi.$$
Observe next that since $R, Z \ge 0$, $A_{ij}X_iX_j \ge
0$ and $Ric \ge 0$,  all terms on the right hand side of (\ref{eq-evolution-Z})
are nonnegative, therefore 
$$Z_t \geq (n-1)\, \Delta Z.$$
Hence, $\hat{Z} := Z - F + \phi$ satisfies the differential inequality
$$\hat{Z}_t \ge (n-1) \Delta \hat{Z} - F_t + (n-1)\, \Delta F + \phi_t - (n-1)\, \Delta\phi$$
and from the choice of $\phi$ and $f$
$$\hat{Z}_t > (n-1)\, \Delta\hat{Z}.$$
Since $\phi(x) \to \infty$ as $x\to \infty$, $\hat{Z}$ attains the minimum inside
a bounded set and by the maximum principle, we have 
$$(\hat{Z}_{\min})_t > 0$$
which implies that 
$$\hat{Z}_{\min}(t) \ge \hat{Z}_{\min}(0)=\phi(0) > 0.$$
We conclude that  $Z \ge F - \phi$ everywhere,  for $t \ge 0$. We now let
$\delta \to 0$ in the choice of $\phi$. This yields
\begin{equation}
\label{equation-max-principle}
Z \ge F > 0, \qquad \mbox{as soon as} \,\, t > 0.
\end{equation}
On the other hand, $Z(g,X) = 0$ at time $t_0 > 0$, at the point where
$R$ attains its maximum, which contradicts (\ref{equation-max-principle}).
This implies $Z(g,X) \equiv 0$ everywhere, for $t < t_0$, and  finishes the
proof of Lemma \ref{lemma-zero-every}.
\end{proof}

\begin{proof}[Proof of Corollary \ref{cor-eternal}]
The result readily   follows from Lemmas 
\ref{lemma-evolution-Z} and  \ref{lemma-zero-every}. Since $Z
\equiv 0$ and since all terms on the right hand side of (\ref{eq-evolution-Z})
are nonnegative, we obtain from  (\ref{eq-evolution-Z}) the identity 
$$\nabla_iX_j = R\, g_{ij}$$
that is, $g$ is a steady soliton. Since $\nabla_iX_j = \nabla_jX_i$
and since our manifold is simply connected,  the vector field
$X$ is   a gradient of a function, which means that the metric $g$ is a
gradient steady soliton.
\end{proof}

\end{document}